\documentclass{amsart}
\usepackage{amsmath, amssymb,verbatim}
\newtheorem{theorem}{Theorem}

\newtheorem{lemma}{Lemma}
\newtheorem{corollary}{Corollary}

\title[ Commuting Toeplitz operators with biharmonic symbols]
{Commuting Toeplitz operators with biharmonic symbols}

\author[Bouhali \and Louhichi \and Yousef]{Aissa Bouhali$^{1}$ \and Issam Louhichi$^{*,2}$ \and Abdelrahman Yousef$^{3}$}

\thanks{Corresponding author: Issam Louhichi. Email: \texttt{ilouhichi@aus.edu}}

\address[1]{Department of Mathematics,  Higher Normal School of Laghouat, Laghouat, Algeria.\newline  Applied Sciences and Didactics Laboratory, Higher Normal School of Laghouat, Laghouat, Algeria.\newline  Pure and applied mathematics Laboratory, University of Laghouat, Laghouat, Algeria.}
\email{aissa.bouhali@ens-lagh.dz}

\address[2]{Department of Mathematics \& Statistics, American University of Sharjah,  Sharjah, United Arab Emirates.}
\email{ilouhichi@aus.edu}
\address[3]{Department of Mathematics \& Statistics, American University of Sharjah,  Sharjah, United Arab Emirates.}
\email{afyousef@aus.edu}

\begin{document}

	\date{\today} 
\begin{abstract} We investigate the commutant problem for Toeplitz operators on the Bergman space of the unit disk whose symbols belong to a subclass of biharmonic functions. We obtain a complete characterization of when two such Toeplitz operators commute. As a consequence, we derive a full description of normal Toeplitz operators with symbols in this class.
\end{abstract}
\maketitle 

\section{Introduction}
Let $\mathbb{C}$ denote the complex plane and let $\mathbb{D}$ be the open unit disk. We equip $\mathbb{D}$ with the normalized Lebesgue area measure $dA=rdr\frac{d\theta}{\pi}$, where $(r,\theta)$ are polar coordinates. The Hilbert space $L^2(\mathbb{D},dA)$ consists of all functions that are square-integrable on $\mathbb{D}$ with respect to this measure. 

The (unweighted) Bergman space, denoted by $L^2_a(\mathbb{D})$, is the closed subspace of $L^2(\mathbb{D},dA)$ formed by analytic functions on $\mathbb{D}$. It is well known that the monomials $\{z^k:k=0,1,2,\dots\}$ constitute an orthogonal basis for $L^2_a(\mathbb{D})$. Since $L^2_a(\mathbb{D})$ is closed in $L^2(\mathbb{D},dA)$, there exists an orthogonal projection from $L^2(\mathbb{D},dA)$ onto $L^2_a(\mathbb{D})$, usually called the Bergman projection. For further background on Bergman spaces and the associated projection operator, see \cite{h}.

For a bounded measurable function $f$ on $\mathbb{D}$. The Toeplitz operator  $T_f$   on $L^2_a(\mathbb{D})$ is defined by  $$T_f(g) = P(fg), \qquad g\in L^2_a(\mathbb{D}),$$ where $P$ denotes the Bergman projection. The function $f$ is called the symbol of $T_f$. Moreover, for $f\in L^{\infty}(\mathbb{D})$,  $T_f$ is bounded and $||T_f||\leq ||f||_{\infty}$.

Toeplitz operators on Bergman spaces have attracted considerable attention over the past several decades due to their rich algebraic structure and connections with operator theory and function theory. A fundamental problem in this area is the description of the commutant of a Toeplitz operator, that is, the set of all Toeplitz operators that commute with it. While substantial progress has been made in several special situations, the general problem of determining when two Toeplitz operators commute, particularly for broad classes of symbols, remains largely unresolved. For an account of the results known so far on the commutativity problem for Toeplitz operators on $L^2_a(\mathbb{D})$, as well as various cases where partial answers have been obtained, we refer the reader to \cite{acr, cao, carl, carl2, c, cr, dong, lt, l, lr, lry, lsz, lz, rv, y}. 

The present work is motivated by the results obtained in \cite{ac}, where the authors characterized commuting Toeplitz operators with bounded harmonic symbols and applied this characterization to describe the normal Toeplitz operators whose symbols are bounded harmonic functions. In the spirit of that work, we investigate the commutativity problem when the symbols belong to a particular class of biharmonic functions. Since every harmonic function is biharmonic, but the converse does not hold in general, the study of Toeplitz operators with biharmonic symbols naturally extends the framework considered in \cite{ac}.

The introduction of biharmonic symbols in the theory of Toeplitz operators was initiated by the third author in \cite{y}. Recall that a function $f$ is called biharmonic if
$$\nabla^4f=\nabla ^2(\nabla ^2f)=0,$$
where $\nabla^2$ denotes the Laplacian operator. By the Almansi representation, any biharmonic function $f$ on a connected domain of $\mathbb{C}$ (in particular $\mathbb{D}$) can be written in the form
$$f(z)=g(z)+|z|^2h(z),$$
where $g$ and $h$ are harmonic functions. Since harmonic functions can be decomposed into analytic and anti-analytic parts, it follows that a biharmonic functions admits a representation of the form
$$f(z)=g_1(z)+\overline{g_2(z)}+|z|^2(h_1(z)+\overline{h_2(z)}),$$
where $g_i$ and $h_i$ are analytic functions. For a comprehensive treatment of polyharmonic (including biharmonic) functions and Almansi representation, readers may refer to \cite{aron}.  In this paper we focus on the case where these analytic components have finite power series expansions, that is, they are polynomials in $z$. Even under this restriction, the commutativity problem turns out to be highly nontrivial. In particular, the techniques developed in \cite{ac} cannot be directly adapted to the present setting, which requires a substantially different approach. Our method relies primarily on the use of the Mellin transform, which provides an effective tool for analyzing the resulting functional identities. We believe that this approach highlights both the motivation and the novelty of the results obtained in this work.

\section{Preliminaries}

	For a measurable radial function $\phi$ on $\mathbb{D}$ (i.e., $\phi(z)=\phi(|z|)$), interpreted as a function in $L^{1}([0,1), rdr)$, the Mellin transform, denoted by $\widehat{\phi}$,  is defined by $$\widehat{\phi}(z)=\int_0^1\phi(r)r^{z-1}dr.$$
It is well-known that for such functions, the Mellin transform is bounded  on the right half-plane $\{z:\Re z\geq 2\}$ and is analytic on $\{z:\Re z> 2\}$. 

 The following lemma, taken from \cite[Lemma 5.3, p. 531]{lsz}, describes the action of a certain class of Toeplitz operators on the elements of the orthogonal basis of $L^2_a(\mathbb{D})$. 
\begin{lemma}\label{mellin}
	Let $k, p, s\in\mathbb{Z}^+_0$. Then
	$$T_{e^{ip\theta}r^s}(z^k)=\frac{2k+2p+2}{2k+p+s+2}z^{k+p}$$ and 
	$$T_{e^{-ip\theta}r^s}(z^k)=\left\{\begin{array}{ll}0&\textrm{ if }\ 0\leq k\leq p-1,\\
		\frac{2k-2p+2}{2k-p+s+2}z^{k-p}&\textrm{ if }\ k\geq p.\end{array}\right.$$
\end{lemma}	



		
\section{Main Result}

Let $\phi_i(z)=f_i(z)+\overline{g_i(z)}$ and $\psi_i(z)=h_i(z)+\overline{k_i(z)}$ with $i=1,2$, where $f_i, g_i, h_i, k_i$ are polynomials in $z$ of degree $N_i, m_i, P_i, q_i$ respectively:
\begin{equation*}
f_i(z)=\sum_{n=0}^{N_i}a_{i,n}z^n,\  g_i(z)=\sum_{n=0}^{m_i}a_{i,-n}z^n,\  h_i(z)=\sum_{n=0}^{P_i}b_{i,n}z^n,\  k_i(z)=\sum_{n=0}^{q_i}b_{i,-n}z^n.
\end{equation*}
We assume that the products  $a_{i,N_i} a_{i,-m_i}b_{i,P_i} b_{i,-q_i}$ are non-zero for  $i=1,2$, ensuring that these degrees are well-defined. Using these components, we define the biharmonic functions $\Phi$ and $\Psi$ as follows: $$\Phi(z)=\phi_1(z)+\vert z\vert^2\phi_2(z)\quad \textrm{and}\quad \Psi(z)=\psi_1(z)+\vert z\vert^2\psi_2(z).$$
The following theorem is the main result of this paper and can be viewed as an analogue of \cite[Theorem 1, p.~2]{ac}. In that work the authors considered bounded harmonic symbols, whereas here we study Toeplitz operators whose symbols are biharmonic functions of the special form described above.

\begin{theorem}\label{thm2}
	The Toeplitz operators $T_{\Phi}$ and $T_{\Psi}$ commute if and only if there exist constants $C_1$ and $C_2$ such that $\Phi=C_1\Psi+C_2$.
\end{theorem}

\section{Key Lemma}
The following lemma is crucial to the proof of our main result. It establishes that if $T_{\Phi}$ and $T_{\Psi}$ commute, then the degrees of the analytic and anti-analytic parts in the Almansi decomposition of their symbols must coincide (i.e., $N_i=P_i$ and $m_i=q_i$ for $i=1,2$).
\begin{lemma}\label{lem1}
If $T_{\Phi}$ commutes with $T_{\Psi}$, then  $N_i=P_i$ and $m_i=q_i$ for $i=1,2$.
\end{lemma}
\begin{proof}
Since $T_{\Phi}T_{\Psi}=T_{\Psi}T_{\Phi}$, we have that  for all $k\geq 0$ 
\begin{align*}
 & \left(T_{\phi_1}T_{\psi_1}+T_{\phi_1}T_{r^2\psi_2}+T_{r^2\phi_2}T_{\psi_1}+T_{r^2\phi_2}T_{r^2\psi_2}\right)(z^k) \\
&= \left(T_{\psi_1}T_{\phi_1}+T_{r^2\psi_2}T_{\phi_1}+T_{\psi_1}T_{r^2\phi_2}+T_{r^2\psi_2}T_{r^2\phi_2}\right)(z^k).
\end{align*}
Thus, for all $k\geq 0$, we have
\begin{align}\label{eq1}
 &\nonumber \big(T_{f_1}T_{\overline{k_1}}+T_{\overline{g_1}}T_{h_1} + T_{f_1}T_{r^2\overline{k_2}}+T_{f_1}T_{r^2h_2}+T_{\overline{g_1}}T_{r^2\overline{k_2}}+T_{\overline{g_1}}T_{r^2h_2} + T_{r^2f_2}T_{\overline{k_1}}\\&\nonumber+T_{r^2f_2}T_{h_1}
 +T_{r^2\overline{g_2}}T_{\overline{k_1}}+T_{r^2\overline{g_2}}T_{h_1} + T_{r^2f_2}T_{r^2\overline{k_2}}+T_{r^2f_2}T_{r^2h_2}+T_{r^2\overline{g_2}}T_{r^2\overline{k_2}}\\&\nonumber+T_{r^2\overline{g_2}}T_{r^2h_2}\big)(z^k) 
 =  \big(T_{\overline{k_1}}T_{f_1}+T_{h_1}T_{\overline{g_1}} + T_{r^2\overline{k_2}}T_{f_1}+T_{r^2h_2}T_{f_1}+T_{r^2\overline{k_2}}T_{\overline{g_1}}\\&\nonumber+T_{r^2h_2}T_{\overline{g_1}} + T_{\overline{k_1}}T_{r^2f_2}+T_{h_1}T_{r^2f_2}
 +T_{\overline{k_1}}T_{r^2\overline{g_2}}+T_{h_1}T_{r^2\overline{g_2}} + T_{r^2\overline{k_2}}T_{r^2f_2}\\&+T_{r^2h_2}T_{r^2f_2}+T_{r^2\overline{k_2}}T_{r^2\overline{g_2}}+T_{r^2h_2}T_{r^2\overline{g_2}}\big)(z^k).
\end{align}
We begin by assuming that $N_1>N_2$ and show that this leads to a contradiction. Under this assumption, we distinguish the following cases:
\begin{itemize}
\item[]{\bf{Case $P_2\geq P_1$:}}
In equation \eqref{eq1}, the term $z^{k+N_1+P_2}$  comes only from
\begin{equation*}
T_{a_{1,N_1}z^{N_1}}T_{b_{2,P_2}r^2z^{P_2}}(z^k)=T_{b_{2,P_2}r^2z^{P_2}}T_{a_{1,N_1}z^{N_1}}(z^k).
\end{equation*}
Since $a_{1,N_1}b_{2,P_2}\neq 0$, we obtain that $T_{a_{1,N_1}z^{N_1}}$ commutes with $T_{b_{2,P_2}r^2z^{P_2}}$. This is a contradiction, since a Toeplitz operator with an analytic symbol commutes only with Toeplitz operators whose symbols are also analytic.
\item[]{\bf{Case $P_2< P_1$}:} If $N_2+P_1>N_1+P_2$ or $N_2+P_1<N_1+P_2$, then we have a contradiction by same argument as in the previous case.
Now, if $N_2+P_1=N_1+P_2$, then for all $k\geq 0$
\begin{eqnarray*}
\left(T_{a_{1,N_1}z^{N_1}}T_{b_{2,P_2}r^2z^{P_2}}+T_{a_{2,N_2}r^2z^{N_2}}T_{b_{1,P_1}z^{P_1}}\right) (z^k)&=&\Big(T_{b_{2,P_2}r^2z^{P_2}}T_{a_{1,N_1}z^{N_1}}\\&+&T_{b_{1,P_1}z^{P_1}}T_{a_{2,N_2}r^2z^{N_2}}\Big)(z^k).
\end{eqnarray*}
Using Lemma \ref{mellin}, for all $k\geq 0$ we obtain that
\begin{align*}
& a_{1,N_1}b_{2,P_2}\frac{2k+2P_2+2}{2k+2P_2+4}+a_{2,N_2}b_{1,P_1}\frac{2k+2P_1+2N_2+2}{2k+2P_1+2N_2+4} \\
&= a_{1,N_1}b_{2,P_2}\frac{2k+2P_2+2N_1+2}{2k+2P_2+2N_1+4}+a_{2,N_2}b_{1,P_1}\frac{2k+2N_2+2}{2k+2N_2+4},
\end{align*}
which is equivalent to
\begin{align*}
& a_{1,N_1}b_{2,P_2}\left[\frac{2k+2P_2+2}{2k+2P_2+4} - \frac{2k+2P_2+2N_1+2}{2k+2P_2+2N_1+4}\right]  \\
&=  a_{2,N_2}b_{1,P_1}\left[\frac{2k+2N_2+2}{2k+2N_2+4} - \frac{2k+2P_1+2N_2+2}{2k+2P_1+2N_2+4}\right].
\end{align*}
By examining the pole $-2P_2-4$ on the left-hand side of the equation above, we deduce that  $N_2=P_2$. Consequently, it follows that $N_1=P_1$ and that $a_{1,N_1}b_{2,P_2}=a_{2,N_2}b_{1,P_1}$.
\end{itemize}
We now assume that $N_1<N_2$. Under this assumption, we have the following cases:
\begin{itemize}
\item[]{\bf{Case $P_1> P_2$}:} In equation \eqref{eq1}, the term $z^{k+N_2+P_1}$ comes only from
\begin{equation*}
T_{a_{2,N_2}r^2z^{N_2}}T_{b_{1,P_1}z^{P_1}}(z^k)= T_{b_{1,P_1}z^{P_1}}T_{a_{2,N_2}r^2z^{N_2}}(z^k).
\end{equation*}
Since Toeplitz operators with analytic symbols commute only with other such operators \cite{acr}, the equation above holds if and only if  
 $a_{2,N_2}b_{1,P_1}=0$. This  contradicts our assumption that $a_{2,N_2}b_{1,P_1}\neq 0$.
\item[]{\bf{Case} $P_1= P_2$:} In equation \eqref{eq1}, the term $z^{k+N_2+P_2}=z^{k+N_2+P_1}$  comes only from
\begin{eqnarray*}
\left(T_{a_{2,N_2}r^2z^{N_2}}T_{b_{1,P_1}z^{P_1}}+T_{a_{2,N_2}r^2z^{N_2}}T_{b_{2,P_2}r^2z^{P_2}}\right)(z^k)=\Big(T_{b_{1,P_1}z^{P_1}}T_{a_{2,N_2}r^2z^{N_2}}\\+T_{b_{2,P_2}r^2z^{P_2}}T_{a_{2,N_2}r^2z^{N_2}}\Big)(z^k).
\end{eqnarray*}
Using Lemma \ref{mellin}, for all $k\geq 0$ we have 
\begin{align*}
&  a_{2,N_2}b_{1,P_1}\left[\frac{2k+2P_1+2N_2+2}{2k+2P_1+2N_2+4}- \frac{2k+2N_2+2}{2k+2N_2+4}\right] \\ 
&= a_{2,N_2}b_{2,P_2}\frac{2k+2N_2+2P_2+2}{2k+2N_2+2P_2+4}\left[\frac{2k+2N_2+2}{2k+2N_2+4}-\frac{2k+2P_2+2}{2k+2P_2+4}\right].
\end{align*}
This equality holds only if $a_{2,N_2}b_{2,P_2}=0$ and $a_{2, N_2}b_{1, P_1}=0$, which contradicts our assumption that $a_{2,N_2}b_{1,P_1}b_{2,P_2}\neq 0$.
\item[]{\bf{Case} $P_2> P_1$:} In equation \eqref{eq1}, the term $z^{k+N_2+P_2}$ comes only from
\begin{equation*}
T_{a_{2,N_2}r^2z^{N_2}}T_{b_{2,P_2}r^2z^{P_2}}(z^k)= T_{b_{2,P_2}r^2z^{P_2}}T_{a_{2,N_2}r^2z^{N_2}}(z^k).
\end{equation*}
Lemma \ref{mellin} implies that all $k\geq 0$, we have
\begin{equation*}
 a_{2,N_2}b_{2,P_2} \frac{2k+2P_2+2}{2k+2P_2+4} = a_{2,N_2}b_{2,P_2} \frac{2k+2N_2+2}{2k+2N_2+4}.
\end{equation*}
Since $a_{2,N_2}b_{2,P_2}\neq 0$, we must have $N_2=P_2$. We now show that this also implies $P_1=N_1$. To this end, we consider the following possibilities:
\begin{itemize}
\item[(1)] If $P_1>N_1$ (so here we have $N_1<P_1<N_2=P_2$), then the term $z^{k+N_2+P_1}$ in equation \eqref{eq1} comes only from
\begin{align*}
& \left(T_{a_{2,N_2}r^2z^{N_2}}T_{b_{1,P_1}z^{P_1}}+\sum_{s=P_1}^{N_2} T_{a_{2,N_2+P_1-s}r^2z^{N_2+P_1-s}}T_{b_{2,s}r^2z^{s}}\right)(z^k) \\
&= \left(T_{b_{1,P_1}z^{P_1}}T_{a_{2,N_2}r^2z^{N_2}}+\sum_{s=P_1}^{N_2} T_{b_{2,s}r^2z^{s}}T_{a_{2,N_2+P_1-s}r^2z^{N_2+P_1-s}}\right)(z^k).
\end{align*}
Using Lemma \ref{mellin}, for all $k\geq 0$, we obtain that 
\begin{align*}
& a_{2,N_2}b_{1,P_1}\frac{2k+2P_1+2N_2+2}{2k+2P_1+2N_2+4}\\&+\frac{2k+2N_2+2P_1+2}{2k+2N_2+2P_1+4}\sum_{s=P_1}^{N_2} a_{2,N_2+P_1-s}b_{2,s}\frac{2k+2s+2}{2k+2s+4} \\
&= a_{2,N_2}b_{1,P_1}\frac{2k+2N_2+2}{2k+2N_2+4}+\frac{2k+2N_2+2P_1+2}{2k+2N_2+2P_1+4}\sum_{s=P_1}^{N_2} a_{2,s}b_{2,N_2+P_1-s}\frac{2k+2s+2}{2k+2s+4},
\end{align*}
which is equivalent to
\begin{align*}
& a_{2,N_2}b_{1,P_1}\left[\frac{2k+2P_1+2N_2+2}{2k+2P_1+2N_2+4}-\frac{2k+2N_2+2}{2k+2N_2+4}\right] \\
&= \frac{2k+2N_2+2P_1+2}{2k+2N_2+2P_1+4}\sum_{s=P_1}^{N_2} \frac{2k+2s+2}{2k+2s+4}\left(a_{2,s}b_{2,N_2+P_1-s}-a_{2,N_2+P_1-s}b_{2,s}\right).
\end{align*}
Observe that, for $s=P_1,\ldots, N_2-1$, the poles $-2s-4$ on the right-hand side of the equation above  do not appear on the left-hand side. Thus, we must have 
 $$a_{2,s}b_{2,N_2+P_1-s}=a_{2,N_2+P_1-s}b_{2,s}, \textrm{ for all }s=P_1,\ldots,N_2-1.$$ Consequently, the previous equation becomes 
\begin{align*}
& a_{2,N_2}b_{1,P_1}\left[\frac{2k+2P_1+2N_2+2}{2k+2P_1+2N_2+4}-\frac{2k+2N_2+2}{2k+2N_2+4}\right] \\
&= \frac{2k+2N_2+2P_1+2}{2k+2N_2+2P_1+4}\cdot \frac{2k+2N_2+2}{2k+2N_2+4}\left(b_{2,P_1}a_{2,N_2}-a_{2,P_1}b_{2,N_2}\right).
\end{align*}
By taking the limit on both sides as $k\to\infty$, we see that $$0=b_{2,P_1}a_{2,N_2}-a_{2,P_1}b_{2,N_2}.$$ Substituting this back into the  equation, we obtain that  $a_{2,N_2}b_{1,P_1}=0$, which contradicts our assumption that  $a_{2,N_2}b_{1,P_1}\neq 0$.

\item[(2)] If $P_1<N_1$ (so here we have $P_1<N_1<N_2=P_2$), then the term $z^{k+N_2+N_1}=z^{k+N_1+P_2}$ in equation \eqref{eq1} comes only from
\begin{align*}
& \left(T_{a_{1,N_1}z^{N_1}}T_{b_{2,N_2}r^2z^{N_2}}+\sum_{s=N_1}^{N_2} T_{a_{2,N_2+N_1-s}r^2z^{N_2+N_1-s}}T_{b_{2,s}r^2z^{s}}\right)(z^k) \\
&= \left(T_{b_{2,N_2}r^2z^{N_2}}T_{a_{1,N_1}z^{N_1}}+\sum_{s=N_1}^{N_2} T_{b_{2,s}r^2z^{s}}T_{a_{2,N_2+N_1-s}r^2z^{N_2+N_1-s}}\right)(z^k).
\end{align*}
Lemma \ref{mellin} implies that fo all $k\geq 0$, we have
\begin{align*}
& a_{1,N_1}b_{2,N_2}\frac{2k+2N_2+2}{2k+2N_2+4}+\frac{2k+2N_2+2N_1+2}{2k+2N_2+2N_1+4}\sum_{s=N_1}^{N_2} a_{2,N_2+N_1-s}b_{2,s}\frac{2k+2s+2}{2k+2s+4}\\
&=  a_{1,N_1}b_{2,N_2}\frac{2k+2N_1+2N_2+2}{2k+2N_1+2N_2+4} \\
&+\frac{2k+2N_2+2N_1+2}{2k+2N_2+2N_1+4}\sum_{s=N_1}^{N_2} a_{2,s}b_{2,N_2+N_1-s}\frac{2k+2s+2}{2k+2s+4}, 
\end{align*}
which is equivalent to
\begin{align*}
& a_{1,N_1}b_{2,N_2}\left[\frac{2k+2N_2+2}{2k+2N_2+4}-\frac{2k+2N_1+2N_2+2}{2k+2N_1+2N_2+4}\right] \\
&= \frac{2k+2N_2+2N_1+2}{2k+2N_2+2N_1+4}\sum_{s=N_1}^{N_2} \frac{2k+2s+2}{2k+2s+4}\left(a_{2,s}b_{2,N_2+N_1-s}- a_{2,N_2+N_1-s}b_{2,s}\right), 
\end{align*}
Observe that, for $s=N_1,\ldots, N_2-1$, the poles $-2s-4$ on the right-hand side of the equation above  do not appear on the left-hand side. Thus, we must have 
$$a_{2,s}b_{2,N_2+N_1-s}=a_{2,N_2+N_1-s}b_{2,s}, \textrm{ for all }s=N_1,\ldots,N_2-1.$$
Consequently, the previous equation becomes
\begin{align*}
& a_{1,N_1}b_{2,N_2}\left[\frac{2k+2N_2+2}{2k+2N_2+4}-\frac{2k+2N_1+2N_2+2}{2k+2N_1+2N_2+4}\right] \\
&= \frac{2k+2N_2+2N_1+2}{2k+2N_2+2N_1+4}\cdot\frac{2k+2N_2+2}{2k+2N_2+4}\left(a_{2,N_2}b_{2,N_1}- a_{2,N_1}b_{2,N_2}\right).
\end{align*}
Similar argument as in the previous situation yields 
 $a_{1,N_1}b_{2,N_2}=0$, which contradicts our assumption that $a_{1,N_1}b_{2,N_2}\neq 0$. 
\end{itemize}
Therefore, we must have $N_1=P_1$.
\end{itemize}
Finally, by taking the adjoint of both sides of equation \eqref{eq1}, the terms with the highest power of $\bar{z}$ become the terms of the highest power of $z$. Applying the same argument as above, with the $m_i$'s and $q_i$'s playing the role of the $N_i$'s and $P_i$'s respectively, we obtain $m_1=q_1$ and $m_2=q_2$.
\end{proof}

\section{Proof of the main result}
If $T_{\Phi}$ commutes with $T_{\Psi}$, then Lemma \ref{lem1} implies 
\begin{equation}
f_i(z)=\sum_{n=0}^{N_i}a_{i,n}z^n,\  g_i(z)=\sum_{n=0}^{m_i}a_{i,-n}z^n,\  h_i(z)=\sum_{n=0}^{N_i}b_{i,n}z^n,\  k_i(z)=\sum_{n=0}^{m_i}b_{i,-n}z^n,
\end{equation}
where $a_{i,N_i}, a_{i,-m_i},b_{i,N_i}$, $b_{i,-m_i}$ are nonzero coefficients for all $i$. Using equation \eqref{eq1}, the terms in $z^k$ come only from
\begin{align*}
&\Bigg[ \sum_{n=0}^{\min(N_1,m_1)} a_{1,n}\overline{b}_{1,-n}T_{z^{n}}T_{\overline{z}^{n}}+\sum_{n=0}^{\min(N_1,m_1)} b_{1,n}\overline{a}_{1,-n}T_{\overline{z}^{n}}T_{z^{n}}\\&+\sum_{n=0}^{\min(N_1,m_2)} a_{1,n}\overline{b}_{2,-n}T_{z^{n}}T_{r^2\overline{z}^{n}} 
+\sum_{n=0}^{\min(N_2,m_1)} b_{2,n}\overline{a}_{1,-n}T_{\overline{z}^{n}}T_{r^2z^{n}}\\& + \sum_{n=0}^{\min(N_2,m_1)} a_{2,n}\overline{b}_{1,-n}T_{r^2z^{n}}T_{\overline{z}^{n}}+\sum_{n=0}^{\min(N_1,m_2)} b_{1,n}\overline{a}_{2,-n}T_{r^2\overline{z}^{n}}T_{z^{n}}\\ 
&+\sum_{n=0}^{\min(N_2,m_2)} a_{2,n}\overline{b}_{2,-n}T_{r^2z^{n}}T_{r^2\overline{z}^{n}}+\sum_{n=0}^{\min(N_2,m_2)} b_{2,n}\overline{a}_{2,-n}T_{r^2\overline{z}^{n}}T_{r^2z^{n}}\Bigg](z^k) \\
&= \Bigg[\sum_{n=0}^{\min(N_1,m_1)} a_{1,n}\overline{b}_{1,-n}T_{\overline{z}^{n}}T_{z^{n}}+\sum_{n=0}^{\min(N_1,m_1)} b_{1,n}\overline{a}_{1,-n}T_{z^{n}}T_{\overline{z}^{n}}\\&+\sum_{n=0}^{\min(N_1,m_2)} a_{1,n}\overline{b}_{2,-n}T_{r^2\overline{z}^{n}}T_{z^{n}} 
+\sum_{n=0}^{\min(N_2,m_1)} b_{2,n}\overline{a}_{1,-n}T_{r^2z^{n}}T_{\overline{z}^{n}}\\& + \sum_{n=0}^{\min(N_2,m_1)} a_{2,n}\overline{b}_{1,-n}T_{\overline{z}^{n}}T_{r^2z^{n}}+\sum_{n=0}^{\min(N_1,m_2)} b_{1,n}\overline{a}_{2,-n}T_{z^{n}}T_{r^2\overline{z}^{n}} \\
&+\sum_{n=0}^{\min(N_2,m_2)} a_{2,n}\overline{b}_{2,-n}T_{r^2\overline{z}^{n}}T_{r^2z^{n}}+\sum_{n=0}^{\min(N_2,m_2)} b_{2,n}\overline{a}_{2,-n}T_{r^2z^{n}}T_{r^2\overline{z}^{n}}\Bigg](z^k).
\end{align*}
Thus for  $k\geq 0$ large enough, Lemma \ref{mellin} implies 
\begin{eqnarray*}
& \displaystyle{\sum_{n=0}^{\min(N_1,m_1)} a_{1,n}\overline{b}_{1,-n}\frac{2k-2n+2}{2k+2} +\sum_{n=0}^{\min(N_1,m_1)} b_{1,n}\overline{a}_{1,-n}\frac{2k+2}{2k+2n+2}} \\ &\displaystyle{+\sum_{n=0}^{\min(N_1,m_2)} a_{1,n}\overline{b}_{2,-n}\frac{2k-2n+2}{2k+4} 
+\sum_{n=0}^{\min(N_2,m_1)} b_{2,n}\overline{a}_{1,-n}\frac{2k+2}{2k+2n+4}}\\&\displaystyle{ + \sum_{n=0}^{\min(N_2,m_1)} a_{2,n}\overline{b}_{1,-n}\frac{2k-2n+2}{2k+4}+\sum_{n=0}^{\min(N_1,m_2)} b_{1,n}\overline{a}_{2,-n}\frac{2k+2}{2k+2n+4} }\\
&\displaystyle{+\sum_{n=0}^{\min(N_2,m_2)} a_{2,n}\overline{b}_{2,-n}\frac{(2k-2n+2)(2k+2)}{(2k+4)^2}+\sum_{n=0}^{\min(N_2,m_2)} b_{2,n}\overline{a}_{2,-n}\frac{(2k+2n+2)(2k+2)}{(2k+2n+4)^2}} \\
&\displaystyle{= \sum_{n=0}^{\min(N_1,m_1)} a_{1,n}\overline{b}_{1,-n}\frac{2k+2}{2k+2n+2} +\sum_{n=0}^{\min(N_1,m_1)} b_{1,n}\overline{a}_{1,-n}\frac{2k-2n+2}{2k+2} }\\&\displaystyle{+\sum_{n=0}^{\min(N_1,m_2)} a_{1,n}\overline{b}_{2,-n}\frac{2k+2}{2k+2n+4} 
+\sum_{n=0}^{\min(N_2,m_1)} b_{2,n}\overline{a}_{1,-n}\frac{2k-2n+2}{2k+4}}\\& \displaystyle{+ \sum_{n=0}^{\min(N_2,m_1)} a_{2,n}\overline{b}_{1,-n}\frac{2k+2}{2k+2n+4}+\sum_{n=0}^{\min(N_1,m_2)} b_{1,n}\overline{a}_{2,-n}\frac{2k-2n+2}{2k+4}} \\
&\displaystyle{+\sum_{n=0}^{\min(N_2,m_2)} a_{2,n}\overline{b}_{2,-n}\frac{(2k+2n+2)(2k+2)}{(2k+2n+4)^2}+\sum_{n=0}^{\min(N_2,m_2)} b_{2,n}\overline{a}_{2,-n}\frac{(2k-2n+2)(2k+2)}{(2k+4)^2}}.
\end{eqnarray*}
For $n=1,\dots, \min(N_2,m_2)$, multiply both sides by $(2k+2n+4)^2$ and then set $k=-n-2$. This yields
\begin{equation}\label{eq2}
a_{2,n}\overline{b}_{2,-n}=b_{2,n}\overline{a}_{2,-n} ,\textrm{ for all } n=1,\dots,\min(N_2,m_2).
\end{equation}
Without lost of generality, we may assume $N_2>N_1$ (refer to the Appendix for the case $N_2\leq N_1$). Then, for  $d=1,\dots,N_2-N_1$, the terms $z^{k+N_2+N_1+d}$ in equation \eqref{eq1} come only from
\begin{eqnarray*}
&\displaystyle{\left(\sum_{s=N_1+d}^{N_2} T_{a_{2,N_2+N_1+d-s}r^2z^{N_2+N_1+d-s}}T_{b_{2,s}r^2z^{s}}\right)(z^k)}\\& = \displaystyle{\left(\sum_{s=N_1+d}^{N_2} T_{b_{2,s}r^2z^{s}}T_{a_{2,N_2+N_1+d-s}r^2z^{N_2+N_1+d-s}}\right)}(z^k),
\end{eqnarray*}
and so Lemma \ref{mellin} implies
\begin{align*}
& \frac{2k+2N_2+2N_1+2d+2}{2k+2N_2+2N_1+2d+4}\sum_{s=N_1+d}^{N_2} a_{2,N_2+N_1+d-s}b_{2,s}\frac{2k+2s+2}{2k+2s+4}\\
&=  \frac{2k+2N_2+2N_1+2d+2}{2k+2N_2+2N_1+2d+4}\sum_{s=N_1+d}^{N_2} a_{2,s}b_{2,N_2+N_1+d-s}\frac{2k+2s+2}{2k+2s+4}. 
\end{align*}
Thus,  for all $s=N_1+d,\dots,N_2$ and   $d=1,\dots,N_2-N_1$, we have $$a_{2,N_2+N_1+d-s}b_{2,s}=a_{2,s}b_{2,N_2+N_1+d-s}.$$  In particular, for $s=N_2$, we have 
\begin{equation}\label{eq3}
a_{2,N_1+d}b_{2,N_2}=a_{2,N_2}b_{2,N_1+d}\mbox{ for all } d=1,\dots,N_2-N_1,
\end{equation} 
or
\begin{equation}\label{extra}
	a_{2,s}b_{2,N_2}=a_{2,N_2}b_{2,s},\mbox{ for all } s=N_1+1,\dots,N_2.
\end{equation} 
Now, for $d=1,\cdots,N_1$, the terms in $z^{k+N_2+d}$  come only from
\begin{align*}
& \Bigg(\sum_{s=N_2-N_1+d}^{N_2} T_{a_{1,N_2+d-s}z^{N_2+d-s}}T_{b_{2,s}r^2z^{s}} +\sum_{s=N_2-N_1+d}^{N_2} T_{a_{2,s}r^2z^{s}}T_{b_{1,N_2+d-s}z^{N_2+d-s}} \\
	&+\sum_{s=d}^{N_2} T_{a_{2,N_2+d-s}r^2z^{N_2+d-s}}T_{b_{2,s}r^2z^{s}} \Bigg)(z^k)\\
	&=\Bigg(\sum_{s=N_2-N_1+d}^{N_2} T_{b_{2,s}r^2z^{s}}T_{a_{1,N_2+d-s}z^{N_2+d-s}} +\sum_{s=N_2-N_1+d}^{N_2} T_{b_{1,N_2+d-s}z^{N_2+d-s}}T_{a_{2,s}r^2z^{s}} \\
	&+\sum_{s=d}^{N_2} T_{b_{2,s}r^2z^{s}}T_{a_{2,N_2+d-s}r^2z^{N_2+d-s}}\Bigg)(z^k).
\end{align*}
Consequently, Lemma \ref{mellin} implies
\begin{align*}
& \sum_{s=N_2-N_1+d}^{N_2} a_{1,N_2+d-s}b_{2,s}\frac{2k+2s+2}{2k+2s+4}\\& +\frac{2k+2N_2+2d+2}{2k+2N_2+2d+4}\sum_{s=N_2-N_1+d}^{N_2} a_{2,s}b_{1,N_2+d-s} \\
	&+\frac{2k+2N_2+2d+2}{2k+2N_2+2d+4}\sum_{s=d}^{N_2} a_{2,N_2+d-s}b_{2,s}\frac{2k+2s+2}{2k+2s+4} \\
	&= \frac{2k+2N_2+2d+2}{2k+2N_2+2d+4}\sum_{s=N_2-N_1+d}^{N_2} a_{1,N_2+d-s}b_{2,s}\\& +\sum_{s=N_2-N_1+d}^{N_2} a_{2,s}b_{1,N_2+d-s}\frac{2k+2s+2}{2k+2s+4} \\
	&+\frac{2k+2N_2+2d+2}{2k+2N_2+2d+4}\sum_{s=d}^{N_2} a_{2,N_2+d-s}b_{2,s}\frac{2k+2N_2+2d-2s+2}{2k+2N_2+2d-2s+4}.
\end{align*}
Thus
\begin{align*}
& \sum_{s=N_2-N_1+d}^{N_2} a_{1,N_2+d-s}b_{2,s}\frac{2k+2s+2}{2k+2s+4}\\& +\frac{2k+2N_2+2d+2}{2k+2N_2+2d+4}\sum_{s=N_2-N_1+d}^{N_2} a_{2,s}b_{1,N_2+d-s} \\
	&+\frac{2k+2N_2+2d+2}{2k+2N_2+2d+4}\sum_{s=d}^{N_2} a_{2,N_2+d-s}b_{2,s}\frac{2k+2s+2}{2k+2s+4} \\
	&= \frac{2k+2N_2+2d+2}{2k+2N_2+2d+4}\sum_{s=N_2-N_1+d}^{N_2} a_{1,N_2+d-s}b_{2,s}\\& +\sum_{s=N_2-N_1+d}^{N_2} a_{2,s}b_{1,N_2+d-s}\frac{2k+2s+2}{2k+2s+4} \\
	&+\frac{2k+2N_2+2d+2}{2k+2N_2+2d+4}\sum_{s=d}^{N_2} a_{2,s}b_{2,N_2+d-s}\frac{2k+2s+2}{2k+2s+4} ,
\end{align*}
For $s=d,\dots, N_2-N_1+d-1$, we examine the poles $-2s-4$ on both sides of the equation and we obtain that $$a_{2,N_2+d-s}b_{2,s}=a_{2,s}b_{2,N_2+d-s}\textrm{ for all }d=1,\dots,N_1.$$ In particular, for $s=d$, we have 
\begin{equation}\label{eq4}
a_{2,N_2}b_{2,d}=a_{2,d}b_{2,N_2},\mbox{ for all } d=1,\dots,N_1.
\end{equation}
Hence the equation above simplifies to
\begin{align*}
& \sum_{s=N_2-N_1+d}^{N_2} a_{1,N_2+d-s}b_{2,s}\frac{2k+2s+2}{2k+2s+4} \\&+\frac{2k+2N_2+2d+2}{2k+2N_2+2d+4}\sum_{s=N_2-N_1+d}^{N_2} a_{2,s}b_{1,N_2+d-s} \\
	&+\frac{2k+2N_2+2d+2}{2k+2N_2+2d+4}\sum_{s=N_2-N_1+d}^{N_2} a_{2,N_2+d-s}b_{2,s}\frac{2k+2s+2}{2k+2s+4} \\
	&= \frac{2k+2N_2+2d+2}{2k+2N_2+2d+4}\sum_{s=N_2-N_1+d}^{N_2} a_{1,N_2+d-s}b_{2,s} \\&+\sum_{s=N_2-N_1+d}^{N_2} a_{2,s}b_{1,N_2+d-s}\frac{2k+2s+2}{2k+2s+4} \\
	&+\frac{2k+2N_2+2d+2}{2k+2N_2+2d+4}\sum_{s=N_2-N_1+d}^{N_2} a_{2,s}b_{2,N_2+d-s}\frac{2k+2s+2}{2k+2s+4}.
\end{align*}
In the last two summations on both sides of the previous equation, consider the index range $N_2-N_1+d \leq s \leq N_2$. For such $s$, we have $d \leq N_2+d-s \leq N_1$ (i.e., the complementary index $N_2+d-s$ falls within $[d, N_1]$). By equation \eqref{eq4}, we obtain $a_{2,N_2+d-s}b_{2,N_2}=a_{2,N_2}b_{2,N_2+d-s}$ for all such $s$. Combining this with equation \eqref{extra} yields $$a_{2,N_2+d-s}b_{2,s}=a_{2,s}b_{2,N_2+d-s} \textrm{ for all } N_2-N_1+d \leq s \leq N_2.$$ Hence the last equation simplifies to

\begin{align*}
& \sum_{s=N_2-N_1+d}^{N_2} a_{1,N_2+d-s}b_{2,s}\frac{2k+2s+2}{2k+2s+4} \\&+\frac{2k+2N_2+2d+2}{2k+2N_2+2d+4}\sum_{s=N_2-N_1+d}^{N_2} a_{2,s}b_{1,N_2+d-s} \\
	&= \frac{2k+2N_2+2d+2}{2k+2N_2+2d+4}\sum_{s=N_2-N_1+d}^{N_2} a_{1,N_2+d-s}b_{2,s} \\&+\sum_{s=N_2-N_1+d}^{N_2} a_{2,s}b_{1,N_2+d-s}\frac{2k+2s+2}{2k+2s+4},
\end{align*}
or
\begin{align*}
& \sum_{s=N_2-N_1+d}^{N_2} \frac{2k+2s+2}{2k+2s+4}\left[a_{1,N_2+d-s}b_{2,s}-a_{2,s}b_{1,N_2+d-s}\right] \\
	&= \frac{2k+2N_2+2d+2}{2k+2N_2+2d+4}\sum_{s=N_2-N_1+d}^{N_2} \left[a_{1,N_2+d-s}b_{2,s}-a_{2,s}b_{1,N_2+d-s}\right].
\end{align*}
Since the poles on both sides of the equation are distinct, we must have that
 $$a_{1,N_2+d-s}b_{2,s}=a_{2,s}b_{1,N_2+d-s} \textrm{ for all } s=N_2-N_1+d,\dots,N_2 \textrm{ and }
d=1,\dots,N_1.$$ In particular, for $s=N_2$, we obtain
\begin{equation}\label{eq5}
a_{1,d}b_{2,N_2}=a_{2,N_2}b_{1,d},\mbox{ for all } d=1,\dots,N_1.
\end{equation}

We now turn our attention to the coefficients of the polynomials $g_i$'s and $k_i$'s. Without loss of generality, assume $m_1<m_2$. Then in equation \eqref{eq1}, for each $d=1,\dots,2m_2-m_1$, the coefficient of $z^{k-m_1-d}$ arises solely from
\begin{equation*}
	\left(T_{\overline{g_1}}T_{r^2\overline{k_2}}+T_{r^2\overline{g_2}}T_{\overline{k_1}}+T_{r^2\overline{g_2}}T_{r^2\overline{k_2}}\right)(z^k)= \left(T_{r^2\overline{k_2}}T_{\overline{g_1}} +T_{\overline{k_1}}T_{r^2\overline{g_2}}+T_{r^2\overline{k_2}}T_{r^2\overline{g_2}}\right)(z^k).
\end{equation*}
Taking the adjoint of both sides and following the same reasoning used to obtain equations \eqref{eq3}-\eqref{eq5} yields
\begin{equation}\label{eq6}
a_{1,-d}b_{2,-m_2}=a_{2,-m_2}b_{1,-d},\mbox{ for all }d=1,\dots,m_1,
\end{equation}
and
\begin{equation}\label{eq8}
a_{2,-d}b_{2,-m_2}=a_{2,-m_2}b_{2,-d},\mbox{ for all }d=1,\dots,m_2.
\end{equation}
It is easy to see that equations \eqref{eq6} and \eqref{eq8} imply 
\begin{equation}\label{extra2} a_{1,-d}b_{2,-d}=a_{2,-d}b_{1,-d},\textrm{ for all }d=1,\dots,m_1.
\end{equation}
Set $C_1=\frac{a_{2,N_2}}{b_{2,N_2}}$. Combining equations \eqref{eq2}, \eqref{eq5}, and \eqref{extra2} gives
\begin{equation}\label{eq7}
a_{1,-d}=\overline{C_1}b_{1,-d},\mbox{ for all } d=1,\dots,m_1.
\end{equation}
Hence, by the established relations between the coefficients, we obtain:
\[
\sum_{n=0}^{N_1} a_{1,n}z^n =a_{1,0}+ C_1\sum_{n=1}^{N_1} b_{1,n}z^n \quad\text{(by equation \eqref{eq5})},
\]
\[
\sum_{n=0}^{m_1} \overline{a}_{1,-n}\overline{z}^{n} =\overline{a}_{1,-0}+ C_1\sum_{n=1}^{m_1} \overline{b}_{1,-n}\overline{z}^{n} \quad\text{(by equation \eqref{eq7})},
\]
\[
r^2\sum_{n=0}^{N_2} a_{2,n}z^n =a_{2,0}|z|^2+ C_1|z|^2\sum_{n=1}^{N_2} b_{2,n}z^n \quad\text{(by equations \eqref{extra} and \eqref{eq4})},
\]
\[
r^2\sum_{n=0}^{m_2} \overline{a}_{2,-n}\overline{z}^{n} =\overline{a}_{2,-0}|z|^2+C_1 |z|^2\sum_{n=1}^{m_2} \overline{b}_{2,-n}\overline{z}^{n}\text{ (by equations (\ref{eq6}, \ref{eq8},\ref{eq7}))}.\]
It follows that the symbols $\Psi$ and $\Phi$ can be written as
	$$\Psi(z)=b_{1,0}+\overline{b}_{1,-0}+(b_{2,0}+\overline{b}_{2,-0})|z|^2+\Omega(z)$$ and 
	\begin{equation}\label{Phi}\Phi(z)=a_{1,0}+\overline{a}_{1,-0}+({a}_{2,0}+\overline{a}_{2,-0})|z|^2+C_1\Omega(z),\end{equation} where 
$$\Omega(z)=\sum_{n=1}^{N_1} b_{1,n}z^n+\sum_{n=1}^{m_1} \overline{b}_{1,-n}\overline{z}^{n}+|z|^2\left(\sum_{n=1}^{N_2} b_{2,n}z^n+ \sum_{n=1}^{m_2} \overline{b}_{2,-n}\overline{z}^{n} \right).$$
Using the facts that every Toeplitz operator commutes with itself and with the identity operator, and that Toeplitz operators with radial symbols commute, the commutation condition $T_{\Phi}T_{\Psi}=T_{\Psi}T_{\Phi}$ reduces to 
$$(a_{2,0}+\overline{a}_{2,-0})T_{r^2}T_{\Omega}+C_1(b_{2,0}+\overline{b}_{2,-0})T_{\Omega}T_{r^2}=C_1(b_{2,0}+\overline{b}_{2,-0})T_{r^2}T_{\Omega}+(a_{2,0}+\overline{a}_{2,-0})T_{\Omega}T_{r^2}.$$
Thus, for all $k\geq 0$, the terms in $z^{k+1}$ come only from
\begin{eqnarray*}&({a}_{2,0}+\overline{a}_{2,-0})T_{r^2}\left(T_{b_{1,1}z}+T_{b_{2,1}r^2z}\right)(z^k)\\&+C_1(b_{2,0}+\overline{b}_{2,-0})\left(T_{b_{1,1}z}+T_{b_{2,1}r^2z}\right)T_{r^2}(z^k)\\
	&=C_1(b_{2,0}+\overline{b}_{2,-0})T_{r^2}\left(T_{b_{1,1}z}+T_{b_{2,1}r^2z}\right)(z^k)\\&+({a}_{2,0}+\overline{a}_{2,-0})\left(T_{b_{1,1}z}+T_{b_{2,1}r^2z}\right)T_{r^2}(z^k).
\end{eqnarray*}	
Applying Lemma \ref{mellin} to evaluate these operator products yields, for all $k\geq 0$,
\begin{eqnarray*}&({a}_{2,0}+\overline{a}_{2,-0})\left(b_{1,1}\frac{2k+4}{2k+6}+b_{2,1}\frac{(2k+4)^2}{(2k+6)^2}\right)+C_1(b_{2,0}+\overline{b}_{2,-0})\left(b_{1,1}\frac{2k+2}{2k+4}+b_{2,1}\frac{2k+2}{2k+6}\right)\\&=
C_1(b_{2,0}+\overline{b}_{2,-0})\left(b_{1,1}\frac{2k+4}{2k+6}+b_{2,1}\frac{(2k+4)^2}{(2k+6)^2}\right)+({a}_{2,0}+\overline{a}_{2,-0})\left( b_{1,1}\frac{2k+2}{2k+4}+b_{2,1}\frac{2k+2}{2k+6}\right).
\end{eqnarray*}	
 Multiplying both sides of the equation by $(2k+6)^2$ and setting $k=-3$ (or multiplying both sides by $k+4$ and setting $k=-4$) gives $b_{2,1}({a}_{2,0}+\overline{a}_{2,-0})=C_1b_{2,1}(b_{2,0}+\overline{b}_{2,-0})$ (or $C_1b_{1,1}(b_{2,0}+\overline{b}_{2,-0})=b_{1,1}({a}_{2,0}+\overline{a}_{2,-0})$). Assuming that $b_{2,1}\neq 0$ (without loss of generality we may assume that at least one of $b_{2,1}$ or $b_{1,1}$ is not zero, otherwise if both are zero, we redo the same argument by looking for the terms in $z^{k+2}$), we conclude that $ ({a}_{2,0}+\overline{a}_{2,-0})=C_1(b_{2,0}+\overline{b}_{2,-0})$. Therefore equation \eqref{Phi} implies
\begin{eqnarray*}
	\Phi(z)&=&a_{1,0}+\overline{a}_{1,-0}+C_1(b_{2,0}+\overline{b}_{2,-0})|z|^2+C_1\Omega(z)\\
	&=&a_{1,0}+\overline{a}_{1,-0}+C_1\Bigg[\sum_{n=1}^{N_1} b_{1,n}z^n+\sum_{n=1}^{m_1} \overline{b}_{1,-n}\overline{z}^{n}\\&+&|z|^2\left(\sum_{n=0}^{N_2} b_{2,n}z^n+ \sum_{n=0}^{m_2} \overline{b}_{2,-n}\overline{z}^{n} \right)\Bigg]\\
	&=&C_1\Psi(z)-C_1(b_{1,0}+\overline{b}_{1,-0})+a_{1,0}+\overline{a}_{1,-0}\\
	&=&C_1\Psi(z)+C_2,
\end{eqnarray*}	
where $C_2=a_{1,0}+\overline{a}_{1,-0}-C_1(b_{1,0}+\overline{b}_{1,-0})$.
\hfill \ensuremath{\Box}

\section{Normal Toeplitz operators}
As a corollary of Theorem \ref{thm2}, we obtain the following characterization of such normal Toeplitz operators. This result serves as an analogue to  \cite[Corollary 17, p.~11]{ac}. We recall that a Toeplitz operator $T_f$ is said to be normal if it commutes with its adjoint, i.e., $T_f T_f^* = T_f^* T_f$. Since the adjoint of a Toeplitz operator is the Toeplitz operator associated with the complex conjugate of the symbol (i.e., $T_f^* = T_{\overline{f}}$), it follows that $T_f$ is normal if and only if $T_f$ commutes with $T_{\overline{f}}$.
\begin{corollary}
Let $\Phi$ be a biharmonic symbol as in Theorem \ref{thm2}. Then $T_{\Phi}$ is a normal operator if and only if $\Phi$ is a constant or the image  $\Phi(\mathbb{D})$ lies on a line in $\mathbb{C}$. 
\end{corollary}
\begin{proof}
$T_{\Phi}$ is normal if and only if $T_{\Phi}T_{\overline{\Phi}}=T_{\overline{\Phi}}T_{\Phi}$. Thus Theorem \ref{thm2} implies that constants $C_1$ and $C_2$ such that $\Phi(z)=C_1\overline{\Phi(z)}+C_2$ for all $z\in\mathbb{D}$. Taking the complex conjugate on both sides gives  $\overline{\Phi(z)}=\overline{C_1}\Phi(z)+\overline{C_2}$. When substituting this into the previous equation, we obtain $(1-|C_1|^2)\Phi(z)=C_1\overline{C_2}+C_2$. Thus if $|C_1|\neq 0$, we have $\Phi(z)=\frac{C_1\overline{C_2}+C_2}{1-|C_1|^2}$ and so $\Phi(z)$ is constant. Otherwise, if $|C_1|=1$, the equation $\Phi(z)=C_1\overline{\Phi(z)}+C_2$, where $C_1=a+ib$ and $C_2=\alpha+i\beta$, yields
\begin{eqnarray*}
	(1-a)\Re[\Phi(z)]-b\Im[\phi(z)]&=&\alpha\\
	-b\Re[\Phi(z)]+(1+a)\Im[\Phi(z)]&=&\beta.
\end{eqnarray*}
 Hence, for every random but fixed $z\in\mathbb{D}$, the pair $(\Re[\Phi(z)],\Im[\phi(z)])$ satisfies a linear relation, which describes a line in the plane. Therefore, for every $z\in\mathbb{D}$, $\Phi(z)$  lies on a fixed line in $\mathbb{C}$. 
 \end{proof}

\section{Appendix}
In the proof of our main result, we assumed $N_2 > N_1$ and obtained the desired conclusion. In what follows, we discuss the cases $N_2 < N_1$ and $N_2 = N_1$ and show that the same result holds. More precisely, we shall show that our key equation, which is equation $\eqref{eq5}$, remains true in those cases.

{\bf{Case $N_2<N_1$:}} For  all $d=1,\dots,N_1-N_2$ and  $k\geq 0$, the terms in $z^{k+2N_2+d}$  of equation \eqref{eq1} come only from
\begin{eqnarray*}
	&\displaystyle{\left(\sum_{s=N_2+d}^{N_1} T_{a_{1,s}z^{s}}T_{b_{2,2N_2+d-s}r^2z^{2N_2+d-s}}+\sum_{s=N_2+d}^{N_1} T_{a_{2,2N_2+d-s}r^{2}z^{2N_2+d-s}}T_{b_{1,s}z^{s}}\right)(z^k)}\\
	& = \displaystyle{\left(\sum_{s=N_2+d}^{N_1} T_{b_{2,2N_2+d-s}r^2z^{2N_2+d-s}}T_{a_{1,s}z^{s}}+\sum_{s=N_2+d}^{N_1} T_{b_{1,s}z^{s}}T_{a_{2,2N_2+d-s}r^{2}z^{2N_2+d-s}}\right)(z^k)}.
\end{eqnarray*}
Then Lemma \ref{mellin} implies  
\begin{align*}
	& \sum_{s=N_2+d}^{N_1} a_{1,s}b_{2,2N_2+d-s}\frac{2k+4N_2+2d-2s+2}{2k+4N_2+2d-2s+4}\\
	&+ \frac{2k+4N_2+2d+2}{2k+4N_2+2d+4}\sum_{s=N_2+d}^{N_1} a_{2,2N_2+d-s}b_{1,s} \\
	&=  \frac{2k+4N_2+2d+2}{2k+4N_2+2d+4}\sum_{s=N_2+d}^{N_1} a_{1,s}b_{2,2N_2+d-s}\\
	&+\sum_{s=N_2+d}^{N_1} a_{2,2N_2+d-s}b_{1,s}\frac{2k+4N_2+2d-2s+2}{2k+4N_2+2d-2s+4} . 
\end{align*}
Thus, for all $s=N_2+d,\dots,N_1$ and  $d=1,\dots,N_1-N_2$, we have
$$a_{1,s}b_{2,2N_2+d-s}=a_{2,2N_2+d-s}b_{1,s}.$$
In particular, for $s=N_2+d$, we obtain 
\begin{equation*}
	a_{1,N_2+d}b_{2,N_2}=a_{2,N_2}b_{1,N_2+d}\quad\textrm{for all } d=1,\dots,N_1-N_2,
\end{equation*} 
or
\begin{equation}\label{eqt01}
	a_{1,s}b_{2,N_2}=a_{2,N_2}b_{1,s}\quad\textrm{for all } s=N_2+1,\dots,N_1.
\end{equation} 
Now, for all $d=1,\cdots,N_2$ and $k\geq 0$, the terms in $z^{k+N_2+d}$ of equation \eqref{eq1}  come only from
\begin{align*}
	& \Bigg(\sum_{s=d}^{N_1} T_{a_{1,s}z^{s}}T_{b_{2,N_2+d-s}r^2z^{N_2+d-s}} +\sum_{s=d}^{N_1} T_{a_{2,N_2+d-s}r^2z^{N_2+d-s}}T_{b_{1,s}z^{s}} \\
	&+\sum_{s=d}^{N_2} T_{a_{2,N_2+d-s}r^2z^{N_2+d-s}}T_{b_{2,s}r^2z^{s}} \Bigg)(z^k)\\
	&=\Bigg(\sum_{s=d}^{N_1} T_{b_{2,N_2+d-s}r^2z^{N_2+d-s}}T_{a_{1,s}z^{s}} +\sum_{s=d}^{N_1} T_{b_{1,s}z^{s}}T_{a_{2,N_2+d-s}r^2z^{N_2+d-s}} \\
	&+\sum_{s=d}^{N_2} T_{b_{2,s}r^2z^{s}} T_{a_{2,N_2+d-s}r^2z^{N_2+d-s}}\Bigg)(z^k).
\end{align*}
Consequently, Lemma \ref{mellin} implies
\begin{align*}
	& \sum_{s=d}^{N_1} a_{1,s}b_{2,N_2+d-s}\frac{2k+2N_2+2d-2s+2}{2k+2N_2+2d-2s+4} \\
	& +\frac{2k+2N_2+2d+2}{2k+2N_2+2d+4}\sum_{s=d}^{N_1} a_{2,N_2+d-s}b_{1,s} \\
	&+\frac{2k+2N_2+2d+2}{2k+2N_2+2d+4}\sum_{s=d}^{N_2} a_{2,N_2+d-s}b_{2,s}\frac{2k+2s+2}{2k+2s+4} \\
	&= \frac{2k+2N_2+2d+2}{2k+2N_2+2d+4}\sum_{s=d}^{N_1} a_{1,s}b_{2,N_2+d-s}\\
	& +\sum_{s=d}^{N_1} a_{2,N_2+d-s}b_{1,s}\frac{2k+2N_2+2d-2s+2}{2k+2N_2+2d-2s+4} \\
	&+\frac{2k+2N_2+2d+2}{2k+2N_2+2d+4}\sum_{s=d}^{N_2} a_{2,N_2+d-s}b_{2,s}\frac{2k+2N_2+2d-2s+2}{2k+2N_2+2d-2s+4}.
\end{align*}
Thus
\begin{align}\label{last one}
	& \sum_{s=N_2-N_1+d}^{N_2} a_{1,N_2+d-s}b_{2,s}\frac{2k+2s+2}{2k+2s+4}\nonumber \\
	& +\frac{2k+2N_2+2d+2}{2k+2N_2+2d+4}\sum_{s=d}^{N_1} a_{2,N_2+d-s}b_{1,s}\nonumber \\
	&+\frac{2k+2N_2+2d+2}{2k+2N_2+2d+4}\sum_{s=d}^{N_2} a_{2,N_2+d-s}b_{2,s}\frac{2k+2s+2}{2k+2s+4}\nonumber \\
	&= \frac{2k+2N_2+2d+2}{2k+2N_2+2d+4}\sum_{s=d}^{N_1} a_{1,s}b_{2,N_2+d-s}\nonumber\\
	& +\sum_{s=N_2-N_1+d}^{N_2} a_{2,s}b_{1,N_2+d-s}\frac{2k+2s+2}{2k+2s+4}\nonumber \\
	&+\frac{2k+2N_2+2d+2}{2k+2N_2+2d+4}\sum_{s=d}^{N_2} a_{2,s}b_{2,N_2+d-s}\frac{2k+2s+2}{2k+2s+4}.
\end{align}
Since $N_2<N_1$, we have  $N_2-N_1+d<d$, so the set $\{s=N_2-N_1+d,\dots,d\}$ is non-empty. Examining the poles $-2s-4$ for each such $s$ on both sides of the last equation yields $$a_{1,N_2+d-s}b_{2,s}=a_{2,s}b_{1,N_2+d-s}\quad \textrm{for all }d=1,\dots,N_2.$$
In particular, for $s=d$, we obtain \begin{equation}\label{second part}
	a_{1,N_2}b_{2,d}=a_{2,d}b_{1,N_2}\quad\textrm{for all }d=1,\dots,N_2.\end{equation}
Combining equations \eqref{eqt01} and \eqref{second part}, we conclude that 
\begin{equation}\label{eqt02}
	a_{2,N_2}b_{2,s}=a_{2,s}b_{2,N_2},\quad  \mbox{for all }s=1,\dots,N_2.
\end{equation}
Using equation \eqref{eqt01}, equation \eqref{last one} reduces to
\begin{align*}
	& \sum_{s=d}^{N_2} a_{1,N_2+d-s}b_{2,s}\frac{2k+2s+2}{2k+2s+4} \\
	&+\frac{2k+2N_2+2d+2}{2k+2N_2+2d+4}\sum_{s=d}^{N_1} a_{2,N_2+d-s}b_{1,s} \\
	&= \frac{2k+2N_2+2d+2}{2k+2N_2+2d+4}\sum_{s=d}^{N_1} a_{1,s}b_{2,N_2+d-s}\\
	& +\sum_{s=d}^{N_2} a_{2,s}b_{1,N_2+d-s}\frac{2k+2s+2}{2k+2s+4}.
\end{align*}
By comparing the poles $-2s-4$ on both sides of the above equation, we must have
$$a_{1,N_2+d-s}b_{2,s}=a_{2,s}b_{1,N_2+d-s} \textrm{ for all }s=d,\cdots,N_2 \textrm{ and } d=1,\cdots,N_2.$$ In particular, for $s=N_2$, we have 
\begin{equation}\label{eq4}
	a_{1,d}b_{2,N_2}=a_{2,N_2}b_{1,d},\mbox{ for all } d=1,\dots,N_2.
\end{equation}
Hence, equations \eqref{eqt01} and \eqref{eq4} imply
\begin{equation}\label{eqt03}
	a_{1,s}b_{2,N_2}=a_{2,N_2}b_{1,s},\mbox{ for all } s=1,\dots,N_1,
\end{equation}
which is exactly equation \eqref{eq5}.\\

{\bf{Case $N_2=N_1$:}} In this case, for all $d=1,\cdots,N_2$ and $k\geq 0$, the terms in $z^{k+N_2+d}$  of equation \eqref{eq1} come only from
\begin{align*}
	& \Bigg(\sum_{s=d}^{N_2} T_{a_{1,s}z^{s}}T_{b_{2,N_2+d-s}r^2z^{N_2+d-s}} +\sum_{s=d}^{N_2} T_{a_{2,N_2+d-s}r^2z^{N_2+d-s}}T_{b_{1,s}z^{s}} \\
	&+\sum_{s=d}^{N_2} T_{a_{2,N_2+d-s}r^2z^{N_2+d-s}}T_{b_{2,s}r^2z^{s}} \Bigg)(z^k)\\
	&=\Bigg(\sum_{s=d}^{N_2} T_{b_{2,N_2+d-s}r^2z^{N_2+d-s}}T_{a_{1,s}z^{s}} +\sum_{s=d}^{N_2} T_{b_{1,s}z^{s}}T_{a_{2,N_2+d-s}r^2z^{N_2+d-s}} \\
	&+\sum_{s=d}^{N_2} T_{b_{2,s}r^2z^{s}} T_{a_{2,N_2+d-s}r^2z^{N_2+d-s}}\Bigg)(z^k).
\end{align*}
Applying Lemma \ref{mellin} yields
\begin{align*}
	& \sum_{s=d}^{N_2} a_{1,s}b_{2,N_2+d-s}\frac{2k+2N_2+2d-2s+2}{2k+2N_2+2d-2s+4} \\
	& +\frac{2k+2N_2+2d+2}{2k+2N_2+2d+4}\sum_{s=d}^{N_2} a_{2,N_2+d-s}b_{1,s} \\
	&+\frac{2k+2N_2+2d+2}{2k+2N_2+2d+4}\sum_{s=d}^{N_2} a_{2,N_2+d-s}b_{2,s}\frac{2k+2s+2}{2k+2s+4} \\
	&= \frac{2k+2N_2+2d+2}{2k+2N_2+2d+4}\sum_{s=d}^{N_2} a_{1,s}b_{2,N_2+d-s}\\
	& +\sum_{s=d}^{N_2} a_{2,N_2+d-s}b_{1,s}\frac{2k+2N_2+2d-2s+2}{2k+2N_2+2d-2s+4} \\
	&+\frac{2k+2N_2+2d+2}{2k+2N_2+2d+4}\sum_{s=d}^{N_2} a_{2,N_2+d-s}b_{2,s}\frac{2k+2N_2+2d-2s+2}{2k+2N_2+2d-2s+4},
\end{align*}
or equivalently 
\begin{align*}
	& \sum_{s=d}^{N_2} a_{1,N_2+d-s}b_{2,s}\frac{2k+2s+2}{2k+2s+4} \\
	& +\frac{2k+2N_2+2d+2}{2k+2N_2+2d+4}\sum_{s=d}^{N_2} a_{2,N_2+d-s}b_{1,s} \\
	&+\frac{2k+2N_2+2d+2}{2k+2N_2+2d+4}\sum_{s=d}^{N_2} a_{2,N_2+d-s}b_{2,s}\frac{2k+2s+2}{2k+2s+4} \\
	&= \frac{2k+2N_2+2d+2}{2k+2N_2+2d+4}\sum_{s=d}^{N_2} a_{1,s}b_{2,N_2+d-s}\\
	& +\sum_{s=d}^{N_2} a_{2,s}b_{1,N_2+d-s}\frac{2k+2s+2}{2k+2s+4} \\
	&+\frac{2k+2N_2+2d+2}{2k+2N_2+2d+4}\sum_{s=d}^{N_2} a_{2,s}b_{2,N_2+d-s}\frac{2k+2s+2}{2k+2s+4}.
\end{align*}
By comparing the poles $-2s-4$ for $d\leq s\leq N_2$ on both sides, we must have
\begin{align}\label{eqt04}
	&\nonumber a_{1,N_2+d-s}b_{2,s} + \frac{-2s-2+2N_2+2d}{-2s+2N_2+2d}a_{2,N_2+d-s}b_{2,s} \\
	&= a_{2,s}b_{1,N_2+d-s} + \frac{-2s-2+2N_2+2d}{-2s+2N_2+2d}a_{2,s}b_{2,N_2+d-s}.
\end{align}
In particular, for $s=d$, we obtain
\begin{equation}\label{before last}
	 a_{1,N_2}b_{2,d} + \frac{2N_2-2}{2N_2}a_{2,N_2}b_{2,d} 
	= a_{2,d}b_{1,N_2} + \frac{2N_2-2}{2N_2}a_{2,d}b_{2,N_2}, 
\end{equation}
and this holds for every $d$ with  $1\leq d\leq N_2$. Thus, for $d=N_2$, we must have
\begin{equation}\label{last}
	a_{1,N_2}b_{2,N_2}= a_{2,N_2}b_{1,N_2}.
\end{equation}
Hence, for all $1\leq d\leq N_2$, equations \eqref{before last} and \eqref{last} imply
\begin{align*}
	 \frac{a_{2,d}}{b_{2,d}}&=\frac{a_{1,N_2}+ \frac{2N_2-2}{2N_2}a_{2,N_2} }{b_{1,N_2} + \frac{2N_2-2}{2N_2}b_{2,N_2}} \\
	&= \frac{\frac{a_{1,N_2}b_{2,N_2}}{a_{2,N_2}} + \frac{2N_2-2}{2N_2}b_{2,N_2}}{\frac{a_{2,N_2}b_{1,N_2}}{b_{2,N_2}} + \frac{2N_2-2}{2N_2}a_{2,N_2} } \\
	&= \frac{a_{2,N_2}}{b_{2,N_2}}.
\end{align*}
Therefore, equation \eqref{eqt04} simplifies to
\begin{equation*}
	a_{1,N_2+d-s}b_{2,s} = a_{2,s}b_{1,N_2+d-s} ,
\end{equation*}
for all $1\leq d\leq s\leq N_2$. Finally, setting $s=N_2$,  we obtain for all $1\leq d\leq N_2$
\begin{equation*}
	a_{1,d}b_{2,N_2} = a_{2,N_2}b_{1,d},
\end{equation*}
which is identical to equation \eqref{eq5}.

\end{document}